\documentclass[10pt]{article}

\usepackage{amsmath}
\usepackage{amssymb}
\usepackage{amsthm,color}
\usepackage{graphics}
\usepackage[latin1]{inputenc}
\usepackage[T1]{fontenc}
\usepackage{comment}
\usepackage{enumerate}
\usepackage{xcolor}

\newtheorem{thm}{Theorem}[section]
\newtheorem{prop}[thm]{Proposition}

\newtheorem{lemma}[thm]{Lemma}
\newtheorem{rem}[thm]{Remark}

\newtheorem{defi}[thm]{Definition}

\newtheorem{conjecture}[thm]{Conjecture}

\newcommand{\R}{\mathbb{R}}             
\newcommand{\N}{\mathbb{N}}             

\newcommand{\Dn}{\Lambda_{\gamma,\lambda} }

\newcommand{\parn}{\par\noindent}

\newcommand{\cea}{c_{\epsilon, \alpha}}
\newcommand{\fea}{f_{\epsilon, \alpha}}
\newcommand{\lea}{\lambda_{\epsilon, \alpha}}

\newcommand{\Section}[1]{\section{#1} \setcounter{equation}{0}}

\begin{document}

\title{Global counterexamples to uniqueness  for a Calder\'on problem with $C^k$ conductivities}
\author{Thierry Daud\'e \footnote{Research supported by the French National Research Project GDR Dynqua} $^{\,1}$, 
Bernard Helffer \footnote{Research supported by the French National Research Project GDR Dynqua} $^{\,2}$, 
Niky Kamran \footnote{Research supported by NSERC grant RGPIN 105490-2018} $^{\,3}$ and Fran\c cois Nicoleau \footnote{Research supported by the French National Research Project GDR Dynqua} $^{\,2}$\\[12pt]
$^1$  \small  Universit\' e de Franche-Comt\' e, CNRS, LmB (UMR 6623), F-25000 Besan\c con, France  \\
\small thierry.daude@univ-fcomte.fr\\
$^2$  \small  Laboratoire de Math\'ematiques Jean Leray, UMR CNRS 6629, \\ 
\small Nantes Universit\'e,  F-44000 Nantes, France.\\
\small Email: bernard.helffer@univ-nantes.fr , francois.nicoleau@univ-nantes.fr \\
$^3$ \small Department of Mathematics and Statistics, McGill University,\\ \small  Montreal, QC, H3A 0B9, Canada. \\
\small Email: niky.kamran@mcgill.ca 
}





\maketitle


\begin{abstract}

Let  $\Omega \subset \R^n$, $n \geq 3$, be a fixed smooth bounded domain, and let $\gamma$ be a smooth conductivity in $\overline{\Omega}$. Consider a non-zero frequency 
$\lambda_0$ which does not belong to the Dirichlet spectrum of $L_\gamma = -{\rm div} (\gamma \nabla \cdot)$. Then, for all $k \geq 1$, there exists an infinite number of pairs of non-isometric $C^k$ conductivities $(\gamma_1, \gamma_2)$ on $\overline{\Omega}$, (see Definition \ref{isomcond}), which are close to $\gamma$ (see Definition \ref{isomclose}) such that the associated DN maps at frequency $\lambda_0$ satisfy
\begin{equation*}
	\Lambda_{\gamma_1,\lambda_0} = \Lambda_{\gamma_2,\lambda_0}.
\end{equation*}


\vspace{0.5cm}




\end{abstract}

\tableofcontents


\Section{Introduction.} \label{0}

\subsection{An anisotropic Calder\' on inverse problem at non-zero frequency.}

\vspace{0.1cm}
\par\noindent 

In this paper, we construct for all $k\geq 1$ global $C^{k}$ counterexamples to uniqueness for a modified version at non-zero frequency of the classical anisotropic Calder\' on inverse problem. Before stating our results, we first need to set up the inverse problem being considered and in doing so introduce some of the notation and terminology that will be used in the rest of our paper. 

\vspace{0.1 cm}

Let $\Omega$  be a bounded domain of $\R^n,\,n \geq 3$, with $C^{\infty}$ boundary, let $\gamma = (\gamma^{ij})$ be a bounded measurable function from $\overline{\Omega}$ to the set ${\mathcal S}_n$ of positive-definite symmetric matrices and let $\lambda \not= 0$ be a non-zero real parameter.  

\vspace{0.1cm}

We consider the following Dirichlet problem:
\begin{equation} \label{Eq00}
	\left\{ \begin{array}{cc}  
		L_\gamma u := - {\rm div\,} (\gamma \nabla u) =  \lambda u     , & \textrm{on} \ \Omega, \\ 
		u = f, & \textrm{on} \ \partial \Omega. 
	\end{array} 
    \right.
\end{equation}

\par
 A classical  result (see for instance \cite{AHG, GT, Sa, Ta1}) ensures that if $\lambda$ does not belong to the Dirichlet spectrum of $L_{\gamma}$, then for any $f \in H^{1/2}(\partial \Omega)$, 
there exists a unique weak solution $u \in H^1(\Omega)$ to the Dirichlet problem (\ref{Eq00}). For completeness, recall that $u \in H^1 (\Omega)$ is a weak solution of (\ref{Eq00}) if
\begin{equation}\label{weaksolution}
	\int_\Omega \gamma \nabla u \cdot \nabla v \ dx   = \lambda  \ 	\int_\Omega u v \ dx \ \  {\rm{for \ all}} \ v \in H_0^1(\Omega),
\end{equation}
and if the trace of $u$ on the boundary is equal to $f$. 

\vspace{0.1cm}

Recall now that the Dirichlet-to-Neumann (DN) map $\Dn: H^{1/2}(\partial \Omega) \to H^{-1/2}(\partial \Omega)$ is the elliptic pseudo-differential operator of order one (at least if  $\gamma$ is regular enough) defined in the weak sense by
\begin{equation}
	\left\langle \Dn f | g \right \rangle = \int_\Omega \gamma \nabla u \cdot \nabla v \ dx  - \lambda \int_{\Omega} uv \ dx , \ {\rm{for \ all}} \ f,  g \in H^{1/2}(\partial \Omega),
\end{equation}
where $u$ is the unique weak solution of the Dirichlet problem (\ref{Eq00}), $v$ is any element of $H^1(\Omega)$ such that $v_{|\partial \Omega} = g$, and
$\left\langle \cdot  | \cdot  \right \rangle$ is the standard $L^2$ duality pairing between $ H^{1/2}(\partial \Omega)$ and its dual. 
In the case in which the coefficient matrix $\gamma$ and the boundary data $f$ are smooth\footnote{Throughout our paper, we say that a function is smooth if it lies in $C^{\infty}(\overline{\Omega})$.}, this definition coincides with the usual one:
\begin{equation}\label{DNmapstrong}
	\Dn f = (\gamma \nabla u) \cdot \nu  _{ \ |\partial \Omega} \  , 
\end{equation}
where $\nu = (\nu^i)$ is the unit outer normal to the boundary. 

\vspace{0.1cm}

The function $\gamma$ will be referred to as a \emph{conductivity} and the parameter $\lambda$ as a $\emph{frequency}$ throughout the paper. This commonly used terminology is motivated by the connections between the DN map $\Dn$ and the voltage-to-current map which is used in electrical impedance tomography (EIT) to recover the properties of a medium, such as its electrical conductivity, from boundary measurements, \cite{U1}. 

\vspace{0.1cm}

Our counterexamples to uniqueness, which will be stated in Section \ref{Main} (see Theorem \ref{Main0}), concern the inverse problem of recovering the conductivity $\gamma = (\gamma^{ij})$ from the knowledge of the DN map $\Dn$ at fixed $\lambda\not=0$, up to a natural gauge equivalence which will be determined next in Section \ref{InvNonzero} below. As we shall explain in that section, this inverse problem, though entirely natural in its formulation, is somewhat different from classical Calder\' on inverse problem~\cite{LeU} in several key aspects related to the nature of the gauge invariance, thus helping to explain the existence of counterexamples to uniqueness for this modified inverse problem which involve  $C^k$ conductivities, and which are global in the sense that the DN map is evaluated on boundary data $f$ whose support consists of the entire boundary $\partial \Omega$.

It should also be noted that while the EIT problem actually corresponds to the case $\lambda=0$, which is excluded by our hypothesis $\lambda\not=0$, the Dirichlet problem (\ref{Eq00}) arises in a number of applications, for example in reflection seismology and inverse obstacle scattering problems for electromagnetic waves with selected frequencies in a inhomogeneous medium, (see \cite{Bao, Ber}). This problem is also closely related to the viscoelasticity wave equation written in the harmonic regime, $u(x)$ being  the scalar displacement field. Inverse problems in viscoelasticity have many applications in medicine and the mechanics of materials, (see for example \cite{BuCh} for details). A related but different problem was also studied in \cite{BeRo, Ou} where the authors show that the self-adjoint Dirichlet (or Robin/mixed) operator associated with an elliptic differential expression on $\Omega$ is determined uniquely up to unitary equivalence by the knowledge of the Dirichlet-to-Neumann maps on an open subset of the boundary and for a set of frequencies which has an accumulation point in the resolvent set of the underlying operator.

\vspace{0.2cm}
\parn

\subsection{Invariance by unimodular diffeomorphisms at non zero frequency and a modified anisotropic Calder\' on conjecture.}\label{InvNonzero}

Although we assumed that the frequency $\lambda \not=0$ when setting up the Dirichlet problem (\ref{Eq00}), both the existence and uniqueness result for solutions of the Dirichlet problem and the definition of the DN map carry over directly to the case $\lambda=0$. It is well-known in that case that the corresponding DN map  $\Lambda_{\gamma,0}$ admits a large gauge invariance group, namely it is invariant under diffeomorphisms $\Psi : \overline{\Omega} \to \overline{\Omega}$ such that $\Psi_{|\partial \Omega} ={\rm Id} $, namely if: 
\begin{equation}\label{pullback}
	\Psi_* \gamma = \left( \frac{  D\Psi \cdot \gamma \cdot  (D\Psi)^T}{|{\rm det\,}  D\Psi |} \right)  \circ \Psi^{-1} ,
\end{equation}
where $D\Psi$ denotes the differential of $\Psi$ and $(D\Psi)^T$ its transpose\footnote{It should be noted that the transformation law (\ref{pullback}), which is that of a $(2,0)$-tensor density of weight $-1$, makes sense even if the components $(\gamma^{ij})$ are only assumed to be bounded and measurable.}, then one has 
\begin{equation}\label{invariance}
	\Lambda_{\Psi_* \gamma,0} = \Lambda_{\gamma,0}\,,
\end{equation}
(see \cite{LeU} and  Lemma \ref{gaugeinvariance} below). We shall see shortly how the gauge invariance (\ref{invariance}) follows as a consequence of Lemma \ref{gaugeinvariance} below. 

\vspace{0.2cm}

This leads to the formulation of the well-known anisotropic Calder\' on conjecture in the case of zero frequency:
\begin{conjecture}[Anisotropic Calder\' on conjecture at zero frequency]\label{CalConj}
Let $\Omega \subset \R^n$, $n \geq 3$, 
be a bounded domain with smooth boundary and let $\gamma_{1},\,\gamma_{2}$ be bounded measurable anisotropic conductivities on $\overline{\Omega}$. 
If 
$$ 
\Lambda_{\gamma_1,0} = 	\Lambda_{\gamma_2,0}
$$
then there exists a diffeomorphism $\Psi: \overline{ \Omega} \to  \overline{ \Omega}$ such that such that $\Psi_{|\partial \Omega} ={\rm Id} $ and such that 
$$
\gamma_2 = \Psi_* \gamma_1\,.
$$
\end{conjecture}

\vspace {0.2cm}

Now, it turns out that when $\lambda \not=0$, there is a corresponding gauge invariance for the DN map $\Dn$ which is a little more subtle than in the case $\lambda=0$. Indeed, as shown below in Lemma \ref{gaugeinvariance}, one has in that case to restrict to the subgroup ${\rm SDiff}(\overline{\Omega})$ of volume-preserving diffeomorphisms of $\Omega$ that are equal to the identity on $\partial \Omega$, i.e  to diffeomorphisms $\Psi$ such that 
$$|{\rm det\,} D\Psi|=1 \mbox{ on } \Omega\,,\quad \Psi_{|\partial \Omega} ={\rm Id} \,.$$ 

\vspace{0.1cm}\parn


\vspace{0.2cm}\parn
This is a consequence of the following:

\begin{lemma}\label{gaugeinvariance}
Let $\Psi : \overline{\Omega} \to \overline{\Omega}$ be a diffeomorphism and  assume that $u$ solves 
$$- {\rm div\,} ((\Psi_* \gamma) \nabla u) = \lambda u\,.$$
Then, if we set  $\tilde{u} = u \circ \Psi$, one has  $$- {\rm div\,} (\gamma \nabla \tilde{u}) = \lambda\, |{\rm det\,} D \Psi | \ \tilde{u}\,.$$
\end{lemma}

\begin{proof}
Let $v \in C_0^{\infty}(\Omega)$ be a test function. We write:
\begin{equation}
	\int_{\Omega} ((\Psi_* \gamma) \nabla_y u)  \cdot  \nabla_y v \ dy =
	\int_{\Omega}  \left( \left( \frac{  D\Psi \cdot \gamma \cdot (D\Psi)^T }{|{\rm det\,}  D\Psi |} \right)  \circ \Psi^{-1} \right) \nabla_y u \cdot   \nabla_y v\ dy .
\end{equation}	
Then, making the change of variables $y = \Psi (x)$, we get immediately  : 
\begin{eqnarray*}
	\int_{\Omega} ((\Psi_* \gamma) \nabla_y u)  \cdot  \nabla_y v \ dy 
	&=&  \int_{\Omega}  \left(  \left( \frac{ D\Psi \cdot \gamma \cdot (D\Psi)^T  }{|{\rm det\,}  D\Psi |} \right) ((D\Psi)^T)^{-1} \nabla_x \tilde{u} \right) 
	\cdot  ((D\Psi)^T)^{-1} \nabla_x \tilde{v} \ |{\rm det\,} D \Psi| \ dx \\
	&=&  \int_{\Omega}  \left(\gamma \nabla_x \tilde{u}\right) \cdot \nabla_x \tilde{v} \ dx 
\end{eqnarray*}
Since by hypothesis $u$ satisfies $- {\rm div\,} ((\Psi_* \gamma) \nabla_y u) = \lambda u$, we have obtained:
\begin{equation}\label{integral}
	\int_{\Omega} \lambda u \ v \ dy = \int_{\Omega}  \left(\gamma \nabla_x \tilde{u}\right) \cdot \nabla_x \tilde{v} \ dx.	
\end{equation}
Making again the change of variables $y = \Psi (x)$ in the left hand side  of (\ref{integral}), we get :
\begin{equation*}
	\int_{\Omega} \left(\lambda  |{\rm det\,} D \Psi | \   \tilde{u} \right)  \ \tilde{v}  \ dx = \int_{\Omega}  \left(\gamma \nabla_x \tilde{u}\right) \cdot \nabla_x \tilde{v} \ dx,	
\end{equation*}
or, in other words,  $- {\rm div\,} (\gamma \nabla \tilde{u}) = \lambda\,  |{\rm det\,} D \psi | \ \tilde{u}$ in a weak sense.
\end{proof}

\vspace{0.2cm}
\parn
When $\Psi_{|\partial \Omega} = {\rm Id}$, it follows from Lemma \ref{gaugeinvariance} that $u$ and $\tilde{u}$ satisfy the same equation (with the same boundary data) 
if and only if  $ |{\rm det\,} D \Psi| =1$ in $\Omega$.

\vspace{0.2cm}
\parn
As a corollary to the preceding lemma, we obtain immediately:
\begin{prop}\label{gaugenonzero}
For any $\lambda \in \R$ and $\Psi \in {\rm SDiff(\overline{ \Omega})}$, we have 
\begin{equation}\label{invfreq} 
	\Lambda_{\Psi_* \gamma,\lambda} = \Dn .
\end{equation}
\end{prop}
\vspace{0.1cm}\parn
In view of the above proposition, we introduce the following definition:

\begin{defi}\label{isomcond} 
 Let $\gamma_1$, $\gamma_2$ be conductivities defined in $\overline{\Omega}$. We say that $\gamma_1$ and $\gamma_2$ are isometric if there exists
  $\Psi \in {\rm SDiff(\overline{ \Omega})}$ such that $\gamma_2 = \Psi_* \gamma_1$.
\end{defi}
The use of the term isometric in the above definition involves a slight abuse of language since conductivities are not tensorial objects akin to metric tensors, as shown by the transformation law (\ref{pullback}). 

\noindent
In the case of non-zero frequency, we are thus led in view of the above discussion to modify the anisotropic Calder\' on conjecture as follows: 
\begin{conjecture}[Modified anisotropic Calder\' on conjecture at non-zero frequency]\label{ModConj}
Let $\Omega \subset \R^n$, $n \geq 3$, 
be a bounded domain with smooth boundary and let $\gamma_{1},\,\gamma_{2}$ be bounded measurable anisotropic conductivities on $\overline{\Omega}$. 
Let $\lambda \not= 0$ be any fixed frequency that does not belong to the Dirichlet spectrum of $L_{\gamma_j},\,j=1,2$.
If 
$$ 
\Lambda_{\gamma_1,\lambda} = 	\Lambda_{\gamma_2,\lambda}
$$
then $\gamma_1$ and $\gamma_2$ are equal up to isometry, that is there exists  $\Psi \in {\rm SDiff(\overline{ \Omega})}$ such that $\gamma_2 = \Psi_* \gamma_1$.
\end{conjecture}
It is appropriate at this stage to make some comments on the equivalent geometric formulation of the anisotropic Calder\' on conjecture at zero frequency in terms of Riemannian metrics as opposed to conductivities \cite{LeU}, and on the ways in which these formulations cease to be equivalent at the level of gauge invariances for the inverse problems once one works at non-zero frequency, as shown in Proposition \ref{gaugenonzero}. These differences will be at the basis of our construction of counterexamples to the anisotropic Calder\' on conjecture at non-zero frequency for smooth conductivities.

\vspace{0.2cm}

Let us begin by observing that we can rewrite the equation 
\begin{equation}\label{Eq000}
L_\gamma u=0\,,
\end{equation}
equivalently in terms of the Laplace-Beltrami operator of a Riemannian metric $(g_{ij})$ on $\overline{\Omega}$ as
\begin{equation}\label{Eq000g}
\Delta_{g}u=0\,,
\end{equation}
where 
$$
 \Delta_{g}=\frac{1}{\sqrt {\det (g_{ij}})}\partial_{i}\bigg(\sqrt{\det (g_{ij})}g^{ij}\partial_{j}\bigg)\,,
$$
and where
\begin{equation}\label{metric}
g^{ij}:=\det(\gamma^{ij})^{\frac{1}{n-2}}\gamma^{ij}\,,\quad g_{ik}\,g^{kj}=\delta^{i}_{j}\,.
\end{equation}
The advantage of rewriting (\ref{Eq00}) in the zero-frequency case $\lambda=0$ in terms of the Riemannian metric whose contravariant components $(g^{ij})$ are given by (\ref{metric}) is that the transformation law (\ref{pullback}) for the conductivity $(\gamma^{ij})$ gets converted into a transformation law for $(g^{ij})$ which is tensorial, that is we have, writing $g$ for the matrix $(g^{ij})$, 
$$
\Psi_* g = \left( D\Psi \cdot g \cdot  (D\Psi)^{T} \right) \circ \Psi^{-1}\,,
$$
for \emph{any} diffeomorphism $\Psi : \overline{\Omega} \to \overline{\Omega}$ such that $\Psi_{|\partial \Omega} ={\rm Id} $. In other words we don't need to assume the unimodularity condition $|D \Psi| =1$ when working at frequency $\lambda =0$ and rewriting (\ref{Eq000}) in the form (\ref{Eq000g}).

\subsection{A brief and non-exhaustive survey of known results on the Calder\' on conjecture. }\label{linktoCalderon}

\vspace{0.1cm}\parn
We now briefly review some of the main contributions to the study of the classical Calder\'{o}n conjecture, that is Conjecture \ref{CalConj}, before stating our results on Conjecture \ref{ModConj}. Note that the main results for this conjecture apply to the global case of full boundary data, that is when ${\mbox {supp}} f = \partial \Omega$, or that of local data, (i.e when the Dirichlet and Neumann data are measured on the same proper open subset $\Gamma$ of the boundary $\partial \Omega$). These contributions are often formulated in terms of the Riemannian metric (\ref{metric}) rather than in terms of conductivities $(\gamma^{ij})$. As remarked earlier, there is no loss of generality in doing this when working at zero frequency.

\vspace{0.1cm}
In dimension $n=2$, for compact and connected surfaces, the anisotropic Calder\'on conjecture (in the smooth case) has been proved for full or local data, (see \cite{LaU, LeU}).

\vspace{0.1cm}
In dimension $n \geq 3$, for real-analytic Riemannian manifolds or for compact connected Einstein manifolds, the anisotropic Calder\'on conjecture has also been settled positively in \cite{GSB, LeU, LaU, LaTU}. 

\vspace{0.1cm}
In the case of smooth rather than analytic metrics, the anisotropic Calder\'on conjecture is still a major open problem, both for full and local data. 
Some important uniqueness results have nevertheless been obtained in the \cite{DSFKSU, DSFKLS, GSB, KS1} for conformally transversally anisotropic manifolds. In constrast, at any fixed frequency $\lambda$ and in the case of partial data measured on \emph{disjoint sets}, the anisotropic conjecture 
has been answered negatively in \cite{DKN2, DKN3, DKN4}.

\vspace{0.1cm}
There are also several important series of papers dealing with the Calder\'on problem for \emph{singular} conductivities.  In dimension $2$, Astala and P\"aiv\"arinta 
showed that an elliptic  isotropic conductivity belonging in $L^{\infty}(\Omega)$ is uniquely determined by the global DN map, (see \cite{ALP, ALP2, AP}). 
In dimension $n \geq3$, Caro and Rogers also established uniqueness in the global Calder\'on problem for elliptic Lipschitz isotropic conductivities \cite{CaRo}.  
In the case of local data, Krupchyk and Uhlmann in \cite{KrUh1} proved that an isotropic conductivity with $\frac{3}{2}$-derivatives  is uniquely determined by a DN map measured on a possibly very small subset of the boundary. There are also important counterexamples to uniqueness due by Greenleaf, Kurylev, Lassas and Uhlmann (see \cite{ALP2, GLU, U2} ) for  metrics which become degenerate along a closed hypersurface.

\vspace{0.1cm}
Finally we mention \cite{DKN2020}, where one shows in dimension $n \geq 3$ that there is non-uniqueness for the 
Calder\'on problem with local data for Riemannian metrics with H\"older continuous coefficients. One constructs a H\"older continuous metric $g$ in a manifold diffeomorphic to a toric cylinder, and shows that there exist in the conformal class of $g$ an infinite number of Riemannian metrics $\tilde{g} = c^4 g$ which are not gauge equivalent to $g$ and for which the DN maps coincide for local data. The corresponding smooth conformal factors are chosen to be harmonic with respect to the metric $g$ and do not satisfy the unique 
continuation principle. We emphasize that this approach cannot be extended to the case of global data precisely because it is required that the frequency $\lambda$ be zero, (see Remark 1.3 in \cite{DKN2020} for an explanation).

\vspace{0.1cm}
 In contrast the main goal of the present paper is to find counterexamples to uniqueness for the {\it{global}} modified Calder\' on conjecture, that is Conjecture \ref{ModConj} at a frequency $\lambda \not=0$.

\Section{Statement of the main result.}\label {Main}

With the above preliminaries at hand, we are now ready to state our main result. Before doing so, it is convenient to introduce the following definition:

\begin{defi}\label{isomclose}
Given $k\geq 0$ and $\epsilon >0$ we say that the conductivities 
are $(\epsilon,k)$-close if  
$$
 ||\gamma_2-\gamma_1 ||_{C^k(\overline{\Omega},\mathcal S_n)} \leq \epsilon\,.
$$
\end{defi}

\vspace{0.3cm}\noindent
The main result of our paper is then the  following : 

\vspace{0.2cm}

\begin{thm} \label{Main0}
         Let $\Omega \subset \R^n, \ n \geq 3$  be a smooth bounded domain and let $\gamma$ be a smooth conductivity in $\overline{\Omega}$. Let us consider 
         $\lambda_0 \neq 0$ which does not belong to the Dirichlet spectrum of $L_\gamma$. 
         Then, for any $k\geq 1$ and $\epsilon >0$
          there exists a pair of non-isometric
         conductivities $(\gamma_1, \gamma_2)$ on $\overline{\Omega}$ of class $C^k$, which are $(\epsilon,k)$ close to $\gamma$ and satisfy
	\begin{equation}
		\Lambda_{\gamma_1,\lambda_0} = \Lambda_{\gamma_2,\lambda_0}.
	\end{equation}
 \end{thm}

In other words, we have found counterexamples to uniqueness for the modified anisotropic Calder\' on problem with conductivities of arbitrary regularity.

\Section{Conformal rescalings of conductivities and adapted diffeomorphisms.}

The non-uniqueness results for the anisotropic Calder\' on problem stated in this paper are based on {\it{both}} usual conformal invariances for Schr\" odinger operators which can be 
written in {\it{div-grad}} form, and on transformations by diffeomorphisms obtained in Lemma \ref{gaugeinvariance}. We have followed more or less the same strategy 
as in \cite{DKN2020} in the Riemannian setting. 

\vspace{0.2cm}
\parn 
In this section, we assume that the conformal factor $c$ and the anisotropic conductivity $\gamma$ are smooth. 
We recall that one has the well-known  conformal identity:
\begin{equation}\label{conformalinv}
	{\rm div\,} (c^2 \gamma \nabla v ) = c \left[ {\rm div\,} (\gamma \nabla (cv)) - {\rm div\,} (\gamma \nabla c) v \right].
\end{equation}	
Thus, if we assume that $v$ satisfies  
\begin{equation}
	- {\rm div\,} (c^2 \gamma \nabla v ) = \lambda v,
\end{equation}
we get immediately 
\begin{equation}\label{conf}
	-  {\rm div\,} (\gamma \nabla (cv)) +  \frac{1}{c}  \left( {\rm div\,} (\gamma \nabla c) + \lambda (c- \frac{1}{c}) \right) (cv ) = \lambda (cv).
\end{equation}

\parn
We can obviously rewrite (\ref{conf}) as
\begin{equation}
	-  {\rm div\,} (\gamma \nabla (cv)) +  \frac{1}{c}  \left( {\rm div\,} (\gamma \nabla c) + \lambda (c- \frac{1}{c} +cf) \right) (cv ) = \lambda (1+f) (cv)\,,
\end{equation}
for any  $f \in C^{\infty}(\overline{\Omega})$. (We will choose $f$ below in order to apply Lemma \ref{gaugeinvariance} with a suitable diffeomorphism $\Psi :  \overline{\Omega} \to \overline{\Omega}$ depending on $f$.) If we assume now that $c$ satisfies
\begin{equation}\label{edpnonlinear}
 {\rm div\,} (\gamma \nabla c) + \lambda (c- \frac{1}{c} +cf) =0\,,
\end{equation}
we get immediately:
\begin{equation}\label{simplification}
-  {\rm div\,} (\gamma \nabla (cv))  = \lambda (1+f) (cv)\,.
\end{equation} 

\vspace{0.1cm}
\parn
Now, we assume that the function $f$ also satisfies for some {\it{fixed}} $\alpha \in (0,1)$,
\begin{equation}\label{propf}
	\int_{\Omega} f(x) \ dx = 0 \ ,  \ 
 ||f||_{0,\alpha} \leq \epsilon \ ,
\end{equation} 
where $\epsilon>0$ is small enough, and  $||  \cdot ||_{k, \alpha}$ denotes the usual norm in the H\"older spaces $C^{k, \alpha} (\overline{\Omega})$. In particular, we see that $1+f \geq \frac{1}{2} \  {\rm{in}} \  \overline{\Omega}$. Under the assumptions (\ref{propf}), 
 there exists for all $ k \in \N$, a $C^{k+1, \alpha}$ diffeomorphism $\Psi :  \overline{\Omega} \to \overline{\Omega}$ such that $\Psi= Id$ on $\partial \Omega$ and 
$|{\rm det} \ D \Psi| = 1+f$ on $\Omega$. Moreover, we have the following estimates:
\begin{equation}\label{normdiffeo}
	|| \ \Psi - {\rm Id} \  ||_{k+1, \alpha} \leq C_{k} \ ||f||_{k, \alpha},
\end{equation}
where the constant $C_{k}$ only depends on $k$ and $\Omega$, (\cite{CDK}, Theorem 10.9, see also \cite{DaMo, RiYe}). 
Thus, using Lemma \ref{gaugeinvariance}, we see that (\ref{simplification}) can be written in the simpler form :
\begin{equation}\label{simplification1}
-  {\rm div\,} (\Psi_* \gamma \nabla w )  = \lambda w\,\, \mbox{ with } \,\, w = (cv) \circ \Psi^{-1}\,. 
\end{equation}

\begin{rem}\label{equiv}
In particular, since the diffeomorphism $\Psi$ restricts to the identity on the boundary $\partial \Omega$, the previous calculation shows that, under the assumptions (\ref{edpnonlinear}) and (\ref{propf}), $\lambda$ is a Dirichlet eigenvalue of the operator  $L_{c^2\gamma}$ if and only if $\lambda$ is a Dirichlet eigenvalue of $L_{\Psi_* \gamma}$. 	
\end{rem}

\vspace{0.1cm}

We therefore immediately get the following result using the definition of the DN map in the smooth case, given in (\ref{DNmapstrong}).

\vspace{0.2cm}
\parn
\begin{prop}\label{smoothcase}
Let $\gamma$ be a smooth conductivity and let $c \in C^{\infty} (\overline{\Omega})$ be a positive conformal factor such that 
\begin{subequations}
\begin{equation} 
 c=1\;, \; \gamma \nabla c \cdot \nu =0 \mbox{ on } \partial \Omega \end{equation}
 and 
\begin{equation}\label{eqnonuniqueness}
{\rm div\,} (\gamma \nabla c) + \lambda (c- \frac{1}{c} +cf)= 0 \ \ {\rm{on}} \ \Omega,
\end{equation}
\end{subequations}
where $f \in C^{\infty}(\overline{\Omega})$ satisfies for some $\alpha \in (0,1)$ and $\epsilon$ small enough,
\begin{equation*}\label{propf1}
	\int_{\Omega} f(x) \ dx = 0 \ ,  \ 
	||f||_{0,\alpha} \leq \epsilon .
\end{equation*} 
Then, there exists for all $k\geq 1$, a $C^{k+1, \alpha}$ diffeomorphism $\Psi : \overline{ \Omega} \to \overline{ \Omega}$ which is equal to the identity on $\partial \Omega$ such that, 
if $\lambda$ is not a Dirichlet eigenvalue of  $L_{c^2\gamma}$, one has:
\begin{equation}\label{eqDNmap}
	\Lambda_{c^2 \gamma,\lambda} = \Lambda_{\Psi_*\gamma,\lambda}.
\end{equation}
\end{prop}

We conclude with the important remark that in the case where $f=0$ or $\lambda=0$, the equality (\ref{eqDNmap}) will not lead to counterexamples to uniqueness since the unique solution of the non linear equation (\ref{eqnonuniqueness}) with the above boundary conditions  is  $c=1$. Indeed, if we set  $ d=c-1$ and 
\begin{equation}
		V = \lambda  \left( \frac{c+1}{c} \right),
\end{equation}
we see that (\ref{eqnonuniqueness})  can be written as
\begin{equation}\label{eqd}
	{\rm div\,} ( \gamma \nabla d ) + Vd  =0 \ \ {\rm{on}} \ \ \Omega,
\end{equation}
 with $d=0$ and $\gamma \nabla d \cdot \nu  =0$ on $\partial \Omega$. Then, it follows from the unique continuation principle, (see for instance \cite{Hor1, Hor2, Tat}),  
 that the unique solution of (\ref{eqd}) is $d=0$, (or equivalently $c=1$). We recall that in dimension $n \geq 3$, the unique continuation property holds for uniformly elliptic operators on a domain $\Omega$ if the coefficients of the principal part of this operator are locally Lipschitz continuous. 
 Note in contrast that in \cite{DKN2020}, a {\it{local}}  nonuniqueness result for the Calder\' on problem was established 
 in the case of a metric with H\"older continuous coefficients.


\section{Proof of Theorem \ref{Main0}.}

\subsection{Numerical range}

\vspace{0.1cm}

Let $\Omega \subset \R^n,\ n \geq 3$, be a smooth bounded domain and let $\gamma$ be a fixed smooth conductivity. Let us begin by an elementary result on the numerical range  with constraints of $L_{\gamma}$. 
In the following, $\lambda_1$ denotes the first eigenvalue of the Dirichlet realization of $L_\gamma$ on $\Omega$. $\lambda_2$ is the second one and we recall that $\lambda_1 < \lambda_2$.
More generally we denote by  $(\lambda_k)$ ($k\geq 1$) the non-decreasing sequence of eigenvalues counted with multiplicity.

\begin{lemma}\label{interval}
	Let $W(L_\gamma)$ the numerical range  with constraints of $L_\gamma$ defined as
	\begin{equation}\label{numerical}
		W(L_\gamma ) = \{ <L_\gamma u , u > \ ; u \in X \},
	\end{equation}
where we have set
\begin{equation}\label{contrainte} 
	X = \{ u \in H_0^1(\Omega, \R)\ ,\ ||u||_2 =1 \ ,\ \int_{\Omega} u(x) \ dx =0 \} \,.
\end{equation}
Then, we have:
\begin{enumerate}
\item $W(L_\gamma)$ is a closed interval in $(\lambda_1,+\infty)$.
	\item More precisely, $W(L_\gamma) = [m, +\infty)$  with $m:= {\rm Inf}\  W(L_\gamma)$ satisfying $\lambda_1 < m\leq \lambda_2$.
\end{enumerate}
\end{lemma}

\begin{proof}
We follow the same strategy as in \cite{Ha,To}. Let $\lambda, \mu \in W(L_\gamma)$ be two distincts points and let $t \in [0,1]$. First, we shall show that $t \lambda + (1-t)\mu  \in  W(L_\gamma )$. 
By definition, there exist $u, v \in X$ such that 
$\lambda = <L_\gamma u, u>$ and $\mu = <L_\gamma v, v>$. Then $u, v$ are linearly independent. It follows that $tu + (1-t)v \not=0$ and we can set:
\begin{equation}
 w_t = \frac{tu + (1-t)v}{||tu + (1-t)v||_2} \in X.
\end{equation}
Let $$S = \alpha\, {\rm  Id } + \beta L_{\gamma}$$ where 
\begin{equation}
 \alpha = - \frac{\mu}{\lambda-\mu} \ \ ,\ \ \beta = \frac{1}{\lambda-\mu}.
\end{equation}
Let $W(S)$ be the numerical range of $S$ with the same form domain $X$.
\vspace{0,1cm}\parn
We claim that if $t \in W(S)$ then $  t\lambda + (1-t) \mu \in W(L_\gamma)$  Indeed, if  $ t \in W(S)$, then there exists $g \in X$ such that 
\begin{equation}
 t = <Sg,g> = \alpha + \beta <L_\gamma g, g>.
\end{equation}
Then, 
\begin{eqnarray}
 t \lambda + (1-t)\mu  &=& (\lambda-\mu) (\alpha + \beta <L_\gamma g, g>) + \mu \nonumber \\
 &=& (-\mu + <L_\gamma g, g> ) + \mu = <L_\gamma g, g> \ \in W(L_\gamma).
\end{eqnarray}

\vspace{0.1cm}\parn
Now, we consider the continuous map $f :[0,1] \to W(S)$ defined by $f(\tau) = <S w_\tau, w_\tau>$. Clearly, one has $f(0)=0$, $f(1)=1$, and using the Intermediate Value Theorem, we get  $t \in [0,1] \subset W(S)$. Thus, $W(L_\gamma)$ is an interval, and we easily see that this interval is closed.

\vspace{0.2cm}\parn
It remains to show that $W(L_\gamma)$ is not bounded. We denote  by $(u_k)$, ($k\geq 1$), an orthonormal basis of eigenfunctions associated with $(\lambda_k)$. For $k\geq 2$, we set
\begin{equation}
 v_k = \frac{1}{\sqrt{1+ \alpha_k^2}} \ (u_k - \alpha_k u_1 ),
\end{equation}
where we have set
\begin{equation}
 \alpha_k = \frac{\int_\Omega u_k (x) \ dx}{ \int_\Omega u_1 (x) \ dx}\,.
\end{equation}
Hence $v_k\in X$ and a  straighforward calculation gives
\begin{equation}
 <L_\gamma v_k , v_k > = \frac{\lambda_k + \alpha_k^2 \lambda_1}{1+ \alpha_k^2} .
\end{equation}
Taking $k=2$ this shows that $m \leq \lambda_2$, with equality if $u_2 \in X$. The inequality $m>\lambda_1$ is a consequence of the fact that the ground state $u_1$ does not vanish in $\Omega$
 hence does not belong to $X$. \\
Considering $k\rightarrow + \infty$, we observe that the sequence $(\alpha_k)$ is bounded and  we get 
\begin{equation}
 \lim_{k \to + \infty} <L_\gamma v_k , v_k >  = + \infty .
\end{equation}
\end{proof}

\begin{rem}\label{Rem4.2}
If we set 
\begin{equation}\label{contraintea}
	X_\infty = \{ u \in C_0^\infty(\Omega, \R)\ ,\ ||u||_2 =1 \ ,\ \int_{\Omega} u(x) \ dx =0 \} \ ,
\end{equation}
then one has:
\begin{equation} (m,+\infty)=\{ <L_\gamma u , u > \ ; u \in X_\infty \},
	\end{equation}
(we use that $X_\infty$ is dense in $X$ and that the set $\{ <L_\gamma u , u > \ ; u \in X_\infty \}$ is an interval).
\end{rem}

\subsection{End of the proof of the main theorem}

\vspace{0.5cm}
\parn
Now, we able to complete the proof of  Theorem \ref{Main0}. Let us consider a fixed $\lambda_0 \not=0$ which does not belong to the Dirichlet spectrum of $L_\gamma$. We have to consider two cases :

\vspace{0.2cm}\parn
{\it Case 1 : Assume that $\lambda_0 >0$.}

\vspace{0.2cm}\parn
Let $\alpha$ be any positive real such that $\frac{\alpha m}{2\alpha+1} < \lambda_0$. Then, using Lemma \ref{interval}, we see that 
$$\lambda_0 \in (\frac{\alpha m}{2\alpha+1}, +\infty) =  W(\frac{\alpha}{2\alpha+1} L_\gamma)\,.
$$
In particular, using also Remark \ref{Rem4.2},  there exists $u \in X_\infty$  such that
\begin{equation}\label{defu}
	\lambda_0 = \frac{\alpha}{2\alpha+1} \ <L_\gamma u, u >.
\end{equation}
Now, for $\epsilon >0$ small enough, we define the positive conformal factor $c_\epsilon (x)$ on $\overline{\Omega}$ by
\begin{equation}\label{cfactor}
	\cea (x) = (1 + \epsilon u(x))^\alpha.
\end{equation}
Clearly, this conformal factor satisfies $\cea (x) =1, \ (\gamma \nabla \cea (x)) \cdot \nu =0$ on $\partial \Omega$. For a suitable frequency $\lambda_{\epsilon, \alpha}>0$ to be defined later, we set:
\begin{equation}\label{deff}
	\fea (x) = -\frac{1}{\lambda_{\epsilon, \alpha} \ \cea} \ {\rm div}  (\gamma \nabla \cea) +  \frac{1}{\cea^2} -1.
\end{equation}
Thus, by construction, the non-linear PDE (\ref{eqnonuniqueness}) is obviously  satisfied and we have $\fea \in C^{\infty}(\overline{\Omega})$. 
Now, we construct the frequency $\lambda_{\epsilon, \alpha}>0$ in order to satisfy
\begin{equation}
	\int_\Omega \fea (x) \ dx = 0.
\end{equation}
We get immediately:
\begin{equation} 
	\frac{1}{\lambda_{\epsilon, \alpha}} \int_\Omega {\rm div}  (\gamma \nabla \cea) \ \frac{1}{ \cea }\ dx \ = \  \int_\Omega  \left( \frac{1}{\cea^2}-1 \right) \ dx,
\end{equation}
or equivalenty using the usual Green formula,
\begin{equation}\label{deflambda}
	\frac{1}{\lambda_{\epsilon, \alpha}} \int_\Omega \frac{\gamma \nabla \cea \cdot \nabla \cea}{c_\epsilon^2} \ dx \ = \ \int_\Omega  \left( \frac{1}{\cea^2} -1 \right) \ dx.
\end{equation}
Now,  since $ u \in X_\infty$, one observes that, when $\epsilon \to 0$, 
\begin{equation}
\int_\Omega  \left( \frac{1}{\cea^2}-1 \right) \ dx   = \alpha (2\alpha+1) \epsilon^2  + O(\epsilon^3),
\end{equation}
and we have
\begin{eqnarray}
\int_\Omega \frac{\gamma \nabla \cea \cdot \nabla \cea}{\cea^2} \ dx &=& \alpha ^2 \epsilon^2 \int_\Omega \gamma \nabla u \cdot \nabla u \ dx + O(\epsilon^3)  \nonumber \\
                                                                    &=& \alpha ^2 \epsilon^2  <L_\gamma  u , u>  + O(\epsilon^3).
\end{eqnarray}
It follows that, for $\epsilon>0$ small enough,  we can define $\lambda_{\epsilon, \alpha}$ as 
\begin{equation}\label{deffrequence}
\lambda_{\epsilon, \alpha} =  \frac{   \int_\Omega \frac{\gamma \nabla \cea \cdot \nabla \cea}{\cea^2} \ dx }{\int_\Omega  \left( \frac{1}{\cea^2}-1 \right) \ dx},
\end{equation}
and we get the asymptotic expansion:
\begin{equation} \label{asymptlambda}
	\lambda_{\epsilon, \alpha}  = \frac{\alpha}{2\alpha +1} \  <L_\gamma  u , u>  + O(\epsilon) = \lambda_0 + O(\epsilon).
	\end{equation}

\parn
In particular, if $\epsilon$ is small enough , we see that $\lambda_{\epsilon, \alpha}  >0$, and we have for all $k \in \N$ and  $\beta \in (0,1)$, $||f_{\epsilon,\alpha}||_{k, \beta} = O(\epsilon)$.
Finally, since the spectrum $\sigma (L_\gamma)$ is discrete and $\cea (x)=1 + O(\epsilon)$, the asymptotics (\ref{asymptlambda}) implies that, for $\epsilon$ small enough, $\lambda$ is not an eigenvalue 
of the operators $L_{\cea \gamma}$, (see Theorem 2.3.3 of  \cite{Hen}).

\vspace{0.2cm}
\parn
As a consequence, using Proposition \ref{smoothcase}, for all $k \in \N$, there exists a $C^{k+1}$ diffeomorphism $\Psi_{\epsilon, \alpha}$ such that :
\begin{equation}\label{eqDNmap1}
	\Lambda_{\cea^2 \gamma,\lea} = \Lambda_{(\Psi_{\epsilon, \alpha})_{*}\gamma,\lea}.
\end{equation}
Moreover, it follows from (\ref{normdiffeo}) and the previous estimates on $c_{\epsilon,\alpha}$ and $f_{\epsilon,\alpha}$ that, for all $k \in \N$,  the conductivities $c_{\epsilon,\alpha}^2 \gamma$ and 
$(\Psi_{\epsilon, \alpha})_{*}\gamma$ are $(\varphi_k (\epsilon),k)$-close to the conductivity $\gamma$, where $\varphi_k (\epsilon) = O(\epsilon)$.

\parn
Now, if we define the new conductivity
\begin{equation}\label{normalisation}
\beta_{\epsilon, \alpha} = \frac{\lambda_0}{\lea} \ \gamma,
\end{equation}
we get obviously:
\begin{equation}\label{eqDNmap3}
	\Lambda_{\cea^2 \beta_{\epsilon, \alpha},\lambda_0} = \Lambda_{(\Psi_{\epsilon, \alpha})_{*} \beta_{\epsilon, \alpha} ,\lambda_0}\,.
\end{equation}
Note that the conductivities $c_{\epsilon, \alpha}^2 \beta_{\epsilon, \alpha}$ and  $(\Psi_{\epsilon, \alpha})_{*} \beta_{\epsilon, \alpha}$ are likewise $(\tilde{\varphi}_k (\epsilon) ,k)$-close with 
$\tilde{\varphi}_k (\epsilon) = O(\epsilon)$.

\vspace{0.5cm}\parn
{\it Case 2 : Assume that $\lambda_0 <0$.}

\vspace{0.2cm}\parn
We follow exactly the same strategy choosing $\alpha \in (-\frac{1}{2}, 0)$ such that $\frac{m\alpha}{2\alpha +1} \in (\lambda_0, 0)$.


\section{On the invariance by isometry.}

Let $\Psi_0 \in {\rm SDiff} ( \overline{\Omega})$ and  let $\gamma_{1}$, $\gamma_{2}$ be two anisotropic conductivities such that $$\gamma_{2} = (\Psi_{0})_ {* }\gamma_{1}\,.$$
 One gets:
\begin{equation}\label{pullback1}
	\gamma_{2} = \left( \frac{  D\Psi_0 \cdot \gamma_{1} \cdot  (D\Psi_0)^T}{|{\rm det\,}  D\Psi_0 |} \right)  \circ \Psi_0^{-1} = 
	\left( D\Psi_0 \cdot \gamma_1 \cdot  (D\Psi_0)^T \right)  \circ \Psi_0^{-1}.
\end{equation}
It follows immediately that:
\begin{equation}\label{eqdet} 
	({\rm det\,}   \gamma_{2}\circ \Psi_0 )^{\frac{1}{n-2}}=  ({\rm det\,}  \gamma_{1})^{\frac{1}{n-2}}   \ \  {\rm{ in}} \ \  \Omega.
\end{equation}

\vspace{0.1cm}\parn
So, if one integrates (\ref{eqdet}) on $\Omega$, and using the change of variables $y = \Psi (x)$, one obtains:
\begin{equation}\label{invariant1}
	\int_{\Omega} ({\rm det\,} {\gamma_{2}})^{\frac{1}{n-2}} \ dy = \int_{\Omega}  ({\rm det\,} \gamma_{1})^{\frac{1}{n-2}} \ dx.
\end{equation}

\vspace{0,2cm}
\parn
Now, let us assume that the conductivities  $\gamma_1:= (\Psi_{\epsilon, \alpha})_{*} \beta_{\epsilon, \alpha}    $ and $\gamma_2:=  c_{\epsilon, \alpha}^2 \beta_{\epsilon, \alpha}    $ are isometric up to  $\Psi_0\in {\rm SDiff}(\overline{\Omega})$. We easily deduce from  (\ref{pullback}), (\ref{normalisation}), (\ref{invariant1}) and a change of variables  that:
\begin{equation}\label{asympt1}
	\int_{\Omega}  ({\rm det\,} \gamma)^{\frac{1}{n-2}}  \ dx = \int_{\Omega}  \cea^{\frac{2n}{n-2}} (x)  \ ({\rm det\,} \gamma)^{\frac{1}{n-2}}  \ dx. 
\end{equation}
Using the following asymptotic expansion:
\begin{equation}
	\cea^{\frac{2n}{n-2}} (x) = 1 + \frac{2\alpha n}{n-2} \ \epsilon \ u(x) + \frac{ \alpha n}{n-2} \left(\frac{2\alpha n}{n-2} -1\right) \ \epsilon^2 u^2(x) + O(\epsilon^3),
\end{equation}
we see, choosing $\alpha \not = \frac{1}{2} - \frac{1}{n}$, that the term of order $2$ gives the relation:
\begin{equation}
	\int_{\Omega} u^2 (x) \ ({\rm det\,} \gamma)^{\frac{1}{n-2}} \  dx =0,
\end{equation}
which is not possible. Thus, the conductivities  $\gamma_1$ and $\gamma_2$ are not isometric up to a volume-preserving diffeomorphism $\Psi_0$.

We conclude this section by observing that we could not have directly invoked the boundary rigidity theorem of Lionheart \cite{Li} for conformal Riemannian structures on manifolds with boundary in order to arrive at the preceding conclusion since we are working with conductivities rather than Riemannian metrics. Lionheart's proof relies critically on the fact, proved using the theory of $G$-structures of finite type, that conformal automorphisms of a Riemannian metric in dimension $n\geq 3$ are determined by their jets of order $2$, \cite{Ko}. We are not aware of the existence of such a result for non-tensorial objects such as conductivities.

\vspace{0.2cm}
As a conclusion,  we have found counterexamples to uniqueness for the modified  Calder\'on  conjecture at a  
non zero frequency.

\begin{rem}
Assume that the conductivity $\gamma$ is of class $C^k$ in $\overline{\Omega}$ with $k \geq 2$, and let us consider a non-zero  frequency $\lambda_0$ which does not belong to the Dirichlet spectrum of $L_\gamma$. It is not difficult to see that, in this case,  there exists an infinite number of pairs of non-isometric $C^{k-2}$ conductivities $(\gamma_1, \gamma_2)$ having the same DN map at the frequency $\lambda_0$.
\end{rem}

\vspace{0.5cm}
\noindent \textbf{Acknowledgements}: The authors would like to warmly thank Dong Ye for fruitful discussions on the prescribed volume form equation in Section 3.



\end{document}